\pdfoutput=1
\newif\ifhs
\hsfalse

\ifhs 
\documentclass{higherStructures}
\usepackage{amssymb}
\usepackage{silence}
\else
\documentclass{amsart}
\usepackage[svgnames]{xcolor}
\usepackage[unicode,
  colorlinks=true,
  linktocpage=true,
  citecolor=ForestGreen,
  linkcolor=MediumOrchid,
  urlcolor=MediumOrchid]{hyperref}
\hypersetup{bookmarksdepth=1} 
\usepackage{enumitem}
\fi

\usepackage{cite}
\usepackage{mathtools}
\usepackage{microtype}
\usepackage{tikz-cd}
\usepackage[capitalise,noabbrev]{cleveref}

\usetikzlibrary{decorations.pathmorphing}

\unless\ifhs
\newtheorem{theorem}{Theorem}[section]
\fi

\newtheorem{proposition}[theorem]{Proposition}
\newtheorem{corollary}[theorem]{Corollary}
\newtheorem{lemma}[theorem]{Lemma}
\theoremstyle{definition}
\newtheorem{definition}[theorem]{Definition}

\newtheorem{notation}[theorem]{Notation}
\theoremstyle{remark}

\newtheorem{remark}[theorem]{Remark}

\title{Operadic categories and 2-Segal sets}
\author{Philip Hackney}
\address{Department of Mathematics, University of Louisiana at Lafayette, USA}
\email{philip@phck.net}

\ifhs
\else
\urladdr{http://phck.net}
\thanks{This work was supported by a grant from the Simons Foundation (\#850849). 
This material is partially based upon work supported by the National Science Foundation under Grant No. DMS-1928930 while the author participated in a program supported by the Mathematical Sciences Research Institute. 
The program was held in the summer of 2022 in partnership with the Universidad Nacional Autónoma de México.}
\fi

\newcommand{\sset}{\mathsf{sSet}}
\newcommand{\cat}{\mathsf{Cat}}
\newcommand{\op}{{\operatorname{op}}}
\newcommand{\oprm}{\op}
\newcommand{\sdrm}{{\operatorname{sd}}}
\newcommand{\actrm}{\textup{act}}
\newcommand{\intrm}{\textup{int}}
\newcommand{\eqrm}{\textup{eq}}

\newcommand{\udec}{\textup{dec}_{\top}}
\newcommand{\ldec}{\textup{dec}_{\perp}}
\newcommand{\sd}{\textup{sd}}
\usepackage{mathrsfs}
\newcommand{\iudecop}{\mathscr{D}^{\top}} 
\newcommand{\ildec}{\mathscr{D}_{\bot}}

\newcommand{\oc}{\mathcal{C}}
\newcommand{\fc}{\mathcal{F}}
\newcommand{\expandC}{\bar\oc}
\newcommand{\ec}{\mathcal{E}}

\newcommand{\id}{\textup{id}}

\newcommand{\lr}[1]{\langle #1 \rangle}

\usepackage{amssymb}
\usepackage[old]{old-arrows}
\usepackage[only,mapsfromchar]{stmaryrd}
\newcommand{\ract}{\varrightarrow\mapsfromchar}

\tikzset{
  act /.tip = >|
}

\begin{document}

\ifhs\maketitle\fi

\begin{abstract}
If $X$ is a 2-Segal set, then the edgewise subdivision of $X$ admits a factorization system coming from upper and lower d\'ecalage.
Using the correspondence between 2-Segal sets and unary operadic categories satisfying the blow-up axiom, the edgewise subdivision of $X$ is interpreted as an enlargement of the associated operadic category, where fiber inclusions have been adjoined.
\end{abstract}
\unless\ifhs\maketitle\fi

\noindent This work is dedicated to Michael Batanin and Martin Markl in celebration of their 60\textsuperscript{th} birthdays.

\section{Introduction}
The operadic categories of Batanin and Markl, which first appeared in \cite{BataninMarkl:OCDDC}, were created to describe operad-like structures.
Namely, each operadic category $\oc$ has an associated notion of `operad' and each such `operad' has its own associated notion of `algebra.'
Operadic come equipped with a cardinality functor sending each object to a finite set, along with an abstract notion of `fibers' of a map, living over the usual fibers of the cardinality of the map.
In contrast to other approaches (e.g. Barwick's operator categories \cite{Barwick:OCHO}), a fiber $\mathfrak{f}_i$ of a map $x\to y$ does not need to be a subobject of $x$, and in fact it may not even be in the same connected component as $x$.
This provides an additional layer of flexibility concerning the types of operad-like structures that fit in the framework.

The theory of operadic categories has been substantially developed in a series of three papers \cite{BataninMarkl:KDOC,BataninMarkl:OCNEKD,BataninMarklObradovic:MMGH}. 
There have also been two papers that have recast or expanded certain aspects.
In \cite{Lack:OCSMCC}, Lack gives a generalization of operadic category, and defines a skew-monoidal category of collections $\mathrm{Coll}_{\oc}$ associated to each generalized operadic category $\oc$. 
Then operads over $\oc$ are described as monoids in $\mathrm{Coll}_{\oc}$.
In \cite{GarnerKockWeber:OCD}, operadic categories and an expanded notion of `lax-operadic categories' are described as algebras for a modified d\'ecalage monad. 

The starting point of this paper is the following source of operadic categories, originally appearing as \cite[Proposition 3.2]{Berger:MCO}.
\begin{theorem}[Berger]\label{berger theorem}
The dual of the active part of a rigid hypermoment category has a canonical structure of operadic category. 
\end{theorem}
\noindent We are interested in the reverse: when is it possible for an operadic category to be canonically expanded into a hypermoment category?
This question was originally posed by Berger.

In this paper we restrict attention to unary operadic categories, that is, those whose cardinality functor is constant at a one-point set.
These are still very interesting, as can be seen in \cite{BataninMarkl:OOCBC} and \cite[Remark 7.2]{Lack:OCSMCC}.
\begin{theorem}[\cref{sec expanded}]
If $\oc$ is a unary operadic category satisfying the blow-up axiom, then there is an expanded category $\expandC$ having the same objects, as well as maps $\mathfrak{f} \leftarrow x$ whenever $x \to y$ is a map in $\oc$ with fiber $\mathfrak{f}$.
\end{theorem}
It is expected that the expanded category $\expandC$ will be an example of a distributive operadic category in the sense of Batanin.
The (weak) blow-up axiom is an essential axiom introduced in \cite{BataninMarkl:OCNEKD}, which allows one to transport activity happening in the fibers of a map $x\to y$ back to $x$.
It plays a key role in the development of Koszul duality theory for operadic categories \cite{BataninMarkl:KDOC}.

The expanded category $\expandC$ is equipped with a \emph{strict factorization system} $(\fc^\op, \oc)$, where $\fc$ is a category of formal inclusions of fibers (\cref{sec fiber inclusions}).
Thus $\expandC$ has an \emph{orthogonal} factorization system, whose right class extends $\oc$ by isomorphisms in $\fc^\op$.
Such an orthogonal factorization system is part of two pieces of data that comprise a hypermoment category.
In \cref{sec hypermoment} we explore the question about whether or not $\expandC^\op$ is a hypermoment category. 
We provide sufficient conditions to guarantee this is the case in \cref{prop reverse berger}.

Our fundamental approach in this paper utilizes simplicial machinery, and in particular the 2-Segal spaces of Dyckerhoff–Kapranov \cite{DyckerhoffKapranov:HSS} / decomposition spaces of G\'{a}lvez-Carrillo--Kock--Tonks \cite{GKT1}.
Ultimately, the preceding theorem (together with the strict factorization system $(\fc^\op, \oc)$) is exactly the same as the following.

\begin{theorem}[\cref{sec canonical}]
If $X$ is a 2-Segal set, then there is a strict factorization system on the edgewise subdivision of $X$ whose left class is isomorphic to $\udec(X)^\op$ and whose right class is isomorphic to $\ldec(X)$.
\end{theorem}

Here, $\udec(X)$ is the upper d\'ecalage of the simplicial set $X$, which is obtained by shifting dimension and discarding the top face and degeneracy maps (and similarly for the lower d\'ecalage $\ldec(X)$).
Meanwhile, the edgewise subdivision $\sd(X)$ only contains odd-dimensional simplices of $X$.
When $X$ is 2-Segal, these three simplicial sets are all (nerves of) categories (see \cref{sec 2-Segal}), so it is sensible to discuss such a strict factorization system.

The connection between the two theorems utilizes results of Garner--Kock--Weber from \cite{GarnerKockWeber:OCD} which say that upper 2-Segal sets are the same thing as unary operadic categories, and 2-Segal sets are the same thing as unary operadic categories satisfying the blow-up axiom.
These equivalences are restrictions of the upper d\'ecalage functor $\udec \colon \sset \to \sset$, which takes upper 2-Segal sets to operadic categories.
As we rely heavily on these correspondences, we provide details in \cref{sec unary as UD} and \cref{app upper dec}.
This is not the only connection between simplicial objects and operadic categories, as there is also the upcoming work of Batanin--Kock--Weber described in \cite{Kock:OCSS}.

\begin{remark}[Motivation]
Our interest in reversibility of \cref{berger theorem} comes from graph categories (see \cite{Hackney:SCGO} for an overview), generalizing the category of trees of Moerdijk and Weiss \cite{MoerdijkWeiss:DS}, which are suitable for describing operad-like structures as well as homotopy-coherent versions.
These graph categories are all hypermoment categories, and several admit rigid variants (see \cite[A.5]{Berger:MCO}). 
This yields important operadic categories from \cite{BataninMarkl:OCNEKD,BataninMarklObradovic:MMGH} via the above construction.
This process throws away inclusions of graphs, which are a key part of the structure of a graph category. 
Subgraphs of graphs are still visible in the associated operadic category, but only as abstract fibers of contractions of subgraphs, and there is no map in the operadic category from the subgraph to the graph representing the inclusion.
Reversing \cref{berger theorem} by formally adjoining fiber inclusions would provide a path to new constructions of graph categories.
Note, however, that graph categories lie far outside the unary case.
\end{remark}

\subsection*{Outline}
The first two sections following the introduction are preliminary matters concerning simplicial objects, where we recall the d\'ecalage and edgewise subdivision constructions, and then relate them to each other and to 2-Segal sets.
\Cref{sec canonical} gives the canonical strict factorization system on the edgewise subdivision of a 2-Segal set.
In \cref{sec unary opcat} we turn to unary operadic categories, recalling their definition, while \cref{sec unary as UD} brings in an alternate characterization.
The category of fiber inclusions arrives in \cref{sec fiber inclusions}, and we characterize isomorphisms in \cref{sec fiber isos}.
Finally, in \cref{sec expanded} we explain the expanded category associated to a unary operadic category satisfying the blow-up axiom, and \cref{sec hypermoment} explores the likelihood of the opposite of the expanded category forming a hypermoment category.

\subsection*{Notation}
We utilize the standard notation $d_i$, $s_i$ for face and degeneracy operators in a simplicial set $X \in \sset$.
Additionally, we write $d_\top \coloneqq d_n \colon X_n \to X_{n-1}$, $s_\top \coloneqq s_n \colon X_n \to X_{n+1}$ for the top face and degeneracy operators, and $d_\bot \coloneqq d_0$, $s_\bot \coloneqq s_0$ for the bottom ones.
In this paper, small categories are not distinguished from their associated simplicial sets and we consider the fully-faithful nerve functor $\cat \to \sset$ as a replete subcategory inclusion.
In other words, we take $\cat$ to be the full subcategory of $\sset$ on those $X$ which satisfy the Grothendieck--Segal condition: the maps
\[ \begin{tikzcd}[row sep=0, column sep=3cm]
X_2 \rar["{d_2, d_0}"] & X_1 \times_{X_0} X_1 \\ 
X_3 \rar["{d_2 d_3, d_0 d_3, d_0 d_0}"] & X_1 \times_{X_0} X_1 \times_{X_0} X_1 \\[-0.2cm] 
\vdots & \vdots \\
X_n \rar[swap]{\prod\limits_{i=1}^n d_\top^{n-i} d_\bot^{i-1}} & X_1 \times_{X_0} X_1 \times_{X_0} \dots  \times_{X_0} X_1
\end{tikzcd} \]
are bijections for every $n \geq 2$.
Alternatively, this holds if and only if the squares
\[ \begin{tikzcd}
X_{n+1} \rar{d_\bot} \dar[swap]{d_\top} & X_n \dar{d_\top} \\
X_n \rar[swap]{d_\bot} & X_{n-1}
\end{tikzcd} \]
are pullbacks for all $n\geq 1$ (see \cite[Lemma 2.10]{GKT1}).
If $X \in \cat$, we will often write $n$-simplices as length $n$ strings of morphisms, leaving the isomorphisms above implicit.

\section{D\'ecalage and edgewise subdivision}\label{sec dec esd}
In this preliminary section, we recall several constructions associated to simplicial objects.
If $X$ is a simplicial set, its upper d\'ecalage $\udec(X)$ is obtained by shifting dimensions down ($\udec(X)_n = X_{n+1}$), and ignoring the top face and degeneracy maps $d_\top, s_\top$.
Similarly, the lower d\'ecalage $\ldec(X)$ shifts dimension down and forgets the bottom face and degeneracy maps $d_\bot = d_0, s_\bot = s_0$.
When $X$ is a category, $\udec(X)$ is the sum of its slices $\udec(X) \cong \sum_{x\in X_0} X_{/x}$ and $\ldec(X)$ is the sum of its coslices.
There is also the edgewise subdivision $\sd(X)$ of a simplicial set (see \cite{BOORS:ESC,KockSpivak:DSST} and \cref{sd X picture}), which recovers the twisted arrow category $\mathrm{tw}(X)$ when $X$ is itself a category.
Moreover, it includes both the upper and lower d\'ecalage inside of it.

\begin{figure}
\[
\begin{tikzcd}
\bullet  \\
\bullet \ar[u, "g"]
\end{tikzcd}
\qquad \qquad
\begin{tikzcd}
&\bullet \\[-0.6cm]
\bullet \ar[ur, "h"] \\
\bullet \ar[u, "g"]\\[-0.6cm]
&\bullet \ar[ul, "f"] \ar[uuu, "k"']
\end{tikzcd}
\qquad \qquad
\begin{tikzcd} 
& & \bullet\\[-0.6cm]
&\bullet \ar[ur,"\ell"] \\[-0.6cm]
\bullet \ar[ur, "h"] \\
\bullet \ar[u, "g"]\\[-0.6cm]
&\bullet \ar[ul, "f"] \ar[uuu, "k"'] \\[-0.6cm]
& & \bullet \ar[ul,"j"] \ar[uuuuu,"m"']
\end{tikzcd}
\]
\caption{In $\sd(X)$: a zero simplex $g$, a one-simplex $g\to k$, and a two-simplex $g\to k \to m$}\label{sd X picture}
\end{figure}
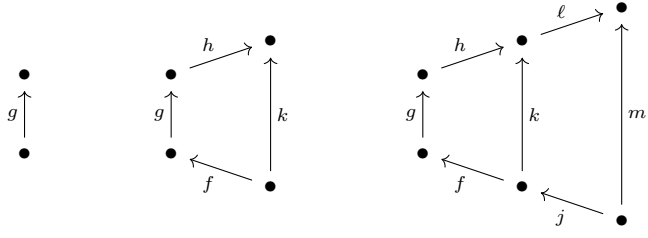

Temporarily write $\dagger \colon \Delta \to \Delta$ for the identity-on-objects automorphism which sends $f \colon [n] \to [m]$ to the function $f^\dagger \colon [n] \to [m]$ with $f^\dagger(t) = m-f(n-t)$.
The opposite of a simplicial set is defined by precomposition with $\dagger$.
Consider the following three endofunctors of $\Delta$:
\begin{align*}
A_\bot \colon  [n] &\mapsto [0]^\oprm \star [n] \cong [n+1] \\
Q \colon  [n] &\mapsto [n]^\oprm \star [n] \cong [2n+1] \\
A_\top \colon  [n] &\mapsto [n]^\oprm \star [0] \cong [n+1]
\end{align*}
Concretely, on $f\colon [n] \to [m]$ these are given by the formulas
\begin{align*}
A_\bot(f)(t) &= \begin{cases}
0 & t=0 \\
f(t-1)+1 & t \geq 1
\end{cases}
&
A_\top(f)(t) &= \begin{cases}
f^\dagger(t) = m - f(n-t) & t \leq n \\
m+1 & t=n+1
\end{cases} \\
& & Q(f)(t) &= \begin{cases}
f^\dagger(t) = m - f(n-t) & t \leq n \\
m+1 + f(t-n-1) & n+1 \leq t.
\end{cases}
\end{align*}
In particular, $Q(f)^\dagger = Q(f)$.

The edgewise subdivision functor $\sd \colon \sset \to \sset$ is defined by precomposition with $Q^\op \colon \Delta^\op \to \Delta^\op$, that is $\sd (X) \coloneqq Q^* X = X \circ Q^\op$.
We have $\sd(X)_n = X_{2n+1}$, and the face maps $\sd(X)_2 \to \sd(X)_1$ and $\sd(X)_1 \to \sd(X)_0$ in low degrees are given in terms of those on $X$ as follows:
\begin{align*}
d_0^\sdrm &= d_2d_3 \colon X_5 \to X_3 & d_0^\sdrm &= d_1d_2 \colon X_3 \to X_1 \\
d_1^\sdrm &= d_1d_4 \colon X_5 \to X_3 & d_1^\sdrm &= d_0d_3 \colon X_3 \to X_1 \\
d_2^\sdrm &= d_0d_5 \colon X_5 \to X_3. 
\end{align*}
Since $\dagger \circ Q = Q$, we have $\sd(X^\op) = \sd(X)$.

Likewise, the lower d\'ecalage functor is defined by $\ldec(X) \coloneqq A_\bot^* X$.
We have $\ldec(X)_n = X_{n+1}$; face and degeneracy maps are shifted by one from those on $X$.
Upper d\'ecalage is pullback along the functor $[n] \mapsto [n]\star [0]$.
There are canonical augmentations $d_\bot \colon \ldec(X) \to X$ and $d_\top \colon \udec(X) \to X$ coming from the unused bottom and top face maps, respectively.

There are natural transformations $A_\bot \Leftarrow Q \Rightarrow A_\top$ which we call $\sigma_\bot$ and $\sigma_\top$, respectively, coming from the unique maps $[0]^\oprm \leftarrow [n]^\oprm$ and $[n] \to [0]$.
Notice also that $A_\top$ is a composite
\[ \begin{tikzcd}[row sep=small]
& \Delta \ar[dr, bend left=20, "(-)\star {[0]}"] & \\
\Delta \ar[ur, bend left=20,"\dagger"] \ar[dr, bend right=20,"A_\bot"'] & & \Delta \\
& \Delta \ar[ur, bend right=20,"\dagger"']
\end{tikzcd} \]
Consequently, for a simplicial set $X$, we have $A_\top^* X = \udec (X)^\op = \ldec (X^\op)$.
The natural transformations $\sigma_\bot$ and $\sigma_\top$ induce maps $A_\bot^* X \to Q^* X \leftarrow A_\top^* X$, which in simplicial degree $n$ are as follows.
\[ \begin{tikzcd}[row sep=tiny]
A_\bot^* X \rar{\sigma_\bot^*} & Q^* X  & \lar[swap]{\sigma_\top^*} A_\top^* X \\
X_{n+1} \rar{s_\bot^n} & X_{2n+1} & \lar[swap]{s_\top^n} X_{n+1}
\end{tikzcd} \]

\begin{notation}
We write $\iudecop (X) \subset \sd (X)$ for the image of $\sigma_\top^*$ and $\ildec (X) \subset \sd (X)$ for the image of $\sigma_\bot^*$.
As simplicial sets, we have $\iudecop (X) \cong \udec (X)^\op$ and $\ildec (X) \cong \ldec (X)$. 
\end{notation}

\section{2-Segal sets / discrete decomposition spaces}\label{sec 2-Segal}

2-Segal spaces, a certain type of simplicial space,  were introduced by Dyckerhoff--Kapranov in \cite{DyckerhoffKapranov:HSS} for applications in representation theory and geometry, and decomposition spaces were introduced in \cite{GKT1} for applications in combinatorics.
These notions turn out to be the same: see \cite{Feller_et_al:E2SSU} for details.
We consider the discrete case, where instead of simplicial spaces we have simplicial sets.
These 2-Segal sets may be considered as a generalization of category where composition is multi-valued \cite[3.3]{DyckerhoffKapranov:HSS}.

\begin{definition}
A simplicial set $X$ is said to be \emph{upper 2-Segal} (resp.\ \emph{lower 2-Segal}) if the square below left (resp.\ below right) is a pullback for all $0 < i < n$.
\[ \begin{tikzcd}
X_{n+1} \rar{d_0} \dar[swap]{d_{i+1}} \ar[dr, phantom, "\lrcorner" very near start] & X_n \dar{d_i} 
& & 
X_{n+1} \rar{d_{n+1}} \dar[swap]{d_i} \ar[dr, phantom, "\lrcorner" very near start] & X_n \dar{d_i}
\\
X_n \rar[swap]{d_0} & X_{n-1}
& & 
X_n \rar[swap]{d_n} & X_{n-1}
\end{tikzcd} \]
A \emph{2-Segal set} is a simplicial set which is both upper and lower 2-Segal.
\end{definition}

Notice that $X$ is upper 2-Segal if and only if $X^\op$ is lower 2-Segal.
By the main theorem of \cite{Feller_et_al:E2SSU}, 2-Segal sets are the same thing as discrete decomposition spaces \cite{GKT1}.

We have the following important \emph{path space criterion}.
It appears in this form as a special case of \cite[Proposition 2.7]{poguntke}; see also \cite[Theorem 6.3.2]{DyckerhoffKapranov:HSS} and \cite[Theorem 4.10]{GKT1}.

\begin{proposition}\label{lem upper 2-segal char}
A simplicial set $X$ is upper 2-Segal if and only if $\udec(X)$ is a category.
Dually, $X$ is lower 2-Segal if and only if $\ldec(X)$ is a category.
\end{proposition}
\begin{proof}
This is simply a slight re-reading of the proof of Proposition 4.9 of \cite{GKT1}.
Namely, if $Y= \udec(X)$ and $n\geq 2$, then the two squares below are equal.
\[ \begin{tikzcd}
Y_{n+1} \rar{d_\bot} \dar[swap]{d_\top} & Y_n \dar{d_\top} & & X_{n+2} \rar{d_0} \dar[swap]{d_{n+1}} & X_{n+1} \dar{d_n}  \\
Y_n \rar[swap]{d_\bot} & Y_{n-1} & & X_{n+1} \rar[swap]{d_0} & X_n
\end{tikzcd} \]
By \cite[Lemma 3.6]{GKT1}, the square on the right is a pullback for all $n\geq 1$ if and only if $X$ is upper 2-Segal.
But $Y = \udec(X)$ is a category if and only if the square on the left is a pullback for all $n \geq 1$.
The dual statement uses $\udec(X^\op) = \ldec(X)^\op$.
\end{proof}

The following is a special case of the edgewise subdivision criterion from \cite{BOORS:ESC}.
\begin{proposition}\label{boors thm}
A simplicial set $X$ is 2-Segal if and only if $\sd(X)$ is a category. \qed
\end{proposition}

These two results reveal a close relationship between $\sd(X)$ and its subobjects $\iudecop (X) \cong \udec(X)^\oprm$ and $\ildec(X) \cong \ldec(X)$.
In the next section, we strengthen this connection.

\section{The canonical factorization system}\label{sec canonical}

We now turn to the canonical (strict) factorization system on the edgewise subdivision of a 2-Segal set. 
Recall that a strict factorization system (see \cite{Grandis:WSEMCC}) on a category $\oc$ is a pair $(L,R)$ of wide subcategories of $\oc$ such that each morphism factors uniquely as $r\circ \ell$ with $\ell \in L$ and $r\in R$.
In particular, $L\cap R$ consists only of identities.
Strict factorization systems are typically not examples of the usual notion of (orthogonal) factorization systems, where $L$ and $R$ are each required to contain all the isomorphisms, and factorizations are unique up to unique isomorphism.

\Cref{prop canonical fact} generalizes an obvious strict factorization system on the twisted arrow category of a category.
\[
\begin{tikzcd}
&\bullet \\[-0.6cm]
\bullet \ar[ur, "h"] \\
\bullet \ar[u, "g"]\\[-0.6cm]
&\bullet \ar[ul, "f"] \ar[uuu, "k"']
\end{tikzcd}
\qquad \qquad
\begin{tikzcd} 
& & \bullet\\[-0.6cm]
&\bullet \ar[ur,"h"] \\[-0.6cm]
\bullet \ar[ur, "\id"] \\
\bullet \ar[u, "g"]\\[-0.6cm]
&\bullet \ar[ul, "f"] \ar[uuu, "gf"'] \\[-0.6cm]
& & \bullet \ar[ul,"\id"] \ar[uuuuu,"k"']
\end{tikzcd}
\]
Namely, each morphism in the twisted arrow category factors uniquely as one whose `forwards map' is an identity, followed by one whose `backwards map' is an identity.

\begin{proposition}\label{prop canonical fact}
If $X \in \sset$ is 2-Segal, then $(\iudecop(X), \ildec(X))$ is a strict factorization system on $\sd (X)$.
\end{proposition}
\begin{proof} Let $Y= \sd(X)$, $L=\iudecop(X)$, and $R=\ildec(X)$.
Since $X$ is lower 2-Segal, the following square is a pullback.
\[ \begin{tikzcd}
X_3 \rar{d_3} \dar[swap]{d_1} \ar[dr, phantom, "\lrcorner" very near start] & X_2 \dar{d_1} \\
X_2 \rar[swap]{d_2} & X_1
\end{tikzcd} \]
By \cref{boors thm}, the square
\[ \begin{tikzcd}
X_5 \rar{d_0d_5} \dar[swap]{d_2d_3} \ar[dr, phantom, "\lrcorner" very near start] & X_3 \dar{d_1d_2}  \ar[drr, phantom, "="]
& &
Y_2 \rar{d_2^\sdrm} \dar[swap]{d_0^\sd} \ar[dr, phantom, "\lrcorner" very near start] & Y_1 \dar{d_0^\sdrm}
\\
X_3 \rar[swap]{d_0d_3} & X_1 & &
Y_1 \rar[swap]{d_1^\sdrm} & Y_0
\end{tikzcd} \]
is a pullback, which implies that
\[ \begin{tikzcd}
X_3 \rar{s_0s_2} \dar["{d_3, d_1}" swap, "\cong"] \ar[dr, phantom, "\lrcorner" very near start] & X_5 \dar["\cong"',"{d_0d_5,d_2d_3}"]  \\
X_2 \times_{X_1} X_2 \rar[swap]{s_\top,s_\bot} & X_3 \times_{X_1} X_3 
\end{tikzcd} \]
is a pullback since both vertical arrows are isomorphisms.
This square is isomorphic to 
\[ \begin{tikzcd}
Y_1 \rar{s_0s_2} \dar[swap]{s_2d_3, s_0d_1} \ar[dr, phantom, "\lrcorner" very near start] & Y_2 \dar["\cong"',"{d_2^\sdrm, d_0^\sdrm}"] \\
L_1 \times_{Y_0} R_1 \rar[hook] & Y_1 \times_{Y_0} Y_1.
\end{tikzcd} \]
Since $s_0s_2 \colon X_3 \to X_5$ is a section of $d_1d_4 = d_1^\sd \colon Y_2 \to Y_1$, every morphism in $Y$ factors uniquely as a morphism in $L$ followed by a morphism in $R$.
\end{proof}

\begin{corollary}\label{second sfs}
If $X \in \sset$ is 2-Segal, then $(\ildec(X), \iudecop(X))$ is a strict factorization system on $\sd (X)$.
\end{corollary}
\begin{proof}
The result is obtained by applying the preceding proposition to $X^\op$: inside $\sd(X) = \sd(X^\op)$ we have $\ildec(X) = \iudecop(X^\op)$ and $\iudecop(X) = \ildec(X^\op)$.
\end{proof}

The map $s_0s_2 \colon X_3 \to X_5$ supplies the factorization in \cref{prop canonical fact}, and the map $s_1s_3 \colon X_3 \to X_5$ supplies the factorization in \cref{second sfs}.

\begin{remark}
This discussion can also be applied to simplicial spaces: if $X$ is a Rezk complete decomposition space (see \cite[5.13]{GKT2}), then $(\iudecop(X), \ildec(X))$ is a factorization system on $\sd(X)$.
Let us explain this result a bit more.
First, under our hypothesis we have by \cite{BOORS:ESC,HK-untwist} that $\sd(X)$ is a Rezk complete Segal space; this type of simplicial space is an important model for $(\infty,1)$-categories \cite{Rezk:MHTHT}.
The maps $\sigma_\bot^* \colon \ldec (X) = A_\bot^* X \to \sd(X)$ and $\sigma_\top^* \colon \udec(X)^\oprm \simeq A_\top^* X \to \sd(X)$ from \cref{sec dec esd} are (levelwise) monomorphisms of simplicial $\infty$-groupoids \cite[Lemma 2.5]{GKT2}.
We will show in the next paragraph that $\ldec(X) \simeq \ildec(X)$ and $\udec(X) \simeq \iudecop(X)^\op$ are also Rezk complete Segal spaces, and the arguments above are readily adapted to establish that $(\iudecop(X), \ildec(X))$ forms an $\infty$-categorical factorization system on $\sd(X)$. 

Recall that an equivalence in a simplicial space $Y$ is an arrow $f\in Y_1$ so that there exist $\alpha, \beta \in Y_2$ with $d_2 \alpha = f =  d_0 \beta$, $d_1\alpha = s_0d_1f$, and $d_1\beta = s_0d_0f$.
We denote by $Y_1^\eqrm \subset Y_1$ the full sub $\infty$-groupoid spanned by the equivalences.
A simplicial space is Rezk complete just when $s_0 \colon Y_0 \to Y_1$ induces an equivalence $Y_0 \simeq Y_1^\eqrm$.
Notice that if $f$ is an equivalence in $\ildec(X)$, then it is also an equivalence in $\sd(X)$.
We thus have the dashed arrow in the following square, which must be a monomorphism.
\[ \begin{tikzcd}
\ildec(X)_0 \rar \dar{\simeq} & \ildec(X)_1^\eqrm \rar[hook] \dar[hook, dashed] & \ildec(X)_1 \dar[hook] \\
\sd(X)_0 \rar{\simeq} & \sd(X)_1^\eqrm \rar[hook] & \sd(X)_1
\end{tikzcd} \]
Since $\sd(X)_0 \simeq \sd(X)_1^\eqrm$ when $X$ is Rezk complete, it follows that $\ildec(X)_0 \to \ildec(X)_1^\eqrm$ is an equivalence.
Further, $\ldec(X)$ is a Segal space by the path space criterion \cite{DyckerhoffKapranov:HSS,GKT1}.
Hence $\ildec(X) \simeq \ldec(X)$ is a Rezk complete Segal space.
A similar argument establishes the result for $\udec(X)$.
\end{remark}

\section{Unary operadic categories}\label{sec unary opcat}
In this section, we discuss unary operadic categories.
The full definition of an operadic category may be found in \cite{BataninMarkl:OCDDC,GarnerKockWeber:OCD}.
Briefly, a unary operadic category is an operadic category where the cardinality functor is identically the one-point set.
We give a self-contained definition below.

A category $\oc$ is said to be \emph{endowed with chosen local terminal objects} if we have specified a terminal object in each connected component of $\oc$.
The letter $u$ will be reserved for these chosen local terminal objects.
If $x$ is some other object of $\oc$, then we write $\tau_x$ or $\tau \colon x\to u$ for the unique map from $x$ to the chosen local terminal object in its connected component.

\begin{definition}
A \emph{unary operadic category} consists of:
\begin{itemize}
\item a category $\oc$ endowed with chosen local terminal objects, along with 
\item a fiber functor $\varphi\colon \udec(\oc) \cong \sum \oc_{/x} \to \oc$.
\end{itemize}
This data is subject to the following three axioms.
\begin{enumerate}[label=(U\arabic*), ref=U\arabic*]
\item For each $x\in \oc$, the object $\varphi_0(\id_x \colon x\to x)$ is a chosen local terminal object.\label{OCA:idfiber}
\item $\varphi_0(\tau_x \colon x \to u) = x$ and $\varphi_1(f \colon \tau_x \to \tau_y) = (f \colon x \to y)$.
\label{OCA:uniquemap}
\item \label{OCA:fibercomp} Suppose we have composable morphisms $g,h$ in $\oc_{/x} \subset \udec(X)$:
\[ \begin{tikzcd}[sep=small]
a \rar["h"] \ar[dr] & b \rar["g"] \dar & c \ar[dl,"f"] \\
& x
\end{tikzcd} \]
Then
$\varphi_0(\varphi_1(fg \xrightarrow{g} f)) = \varphi_0(b \xrightarrow{g} c)$ 
and 
\[
    \varphi_1\left( 
\begin{tikzcd}[column sep=small]
a \ar[rr,"h"] \ar[dr,"gh"'] & & b \ar[dl,"g"] \\
& c
\end{tikzcd}
    \right) =
    \varphi_1\left( 
\begin{tikzcd}[column sep=-0.2cm, row sep=small]
\varphi_0(fgh) \ar[rr,"\varphi_1(h)"] \ar[dr] & & \varphi_0(fg) \ar[dl,"\varphi_1(g)"] \\
& \varphi_0(f)
\end{tikzcd}
    \right).
\]
\end{enumerate}
The category of unary operadic categories consists of functors which both preserve chosen local terminal objects and commute with fiber functors.
\end{definition}

This amounts to an operadic category whose cardinality functor is identically the one-point set.
Comparing to \cite[Definition 1]{GarnerKockWeber:OCD}, the three axioms \eqref{OCA:idfiber}, \eqref{OCA:uniquemap}, and \eqref{OCA:fibercomp} correspond to (A2), (A4), and (A5), while (A1) and (A3) are automatic in the unary case.

We will use tightly dotted lines to diagrammatically express the fiber relationship.
For example, the diagrams below both reflect the situation where $a = \varphi_0(h) = \varphi_0(fg)$, $b=\varphi_0(f)$, and $k= \varphi_1(g\colon h \to f)$.
\[\begin{tikzcd}[row sep=small, column sep=tiny]
a \ar[rr,"k"] \dar[no head, densely dotted] & & b \dar[no head, densely dotted]\\
x \ar[rr,"g"] \ar[dr,"h"'] & & y \ar[dl,"f"]  \\
& z
\end{tikzcd} 
\qquad
\begin{tikzcd}[sep=small]
a \rar{k} \dar[no head, densely dotted] & b \dar[no head, densely dotted]\\
 x \rar["g"] & y \rar["f"] & z
\end{tikzcd}
\]
Likewise, the picture below depicts the situation of \eqref{OCA:fibercomp}.
\[
\begin{tikzcd}[sep=small]
m \dar[no head, densely dotted] \rar{k} & n \dar[no head, densely dotted]
\\
w \dar[no head, densely dotted] \rar{q} & y \rar{p} \dar[no head, densely dotted] & z \dar[no head, densely dotted]\\
a \rar["h"] & b \rar["g"] & c \rar["f"] & x
\end{tikzcd}
\]
Here, $p=\varphi_1(g\colon fg\to f)$, $q=\varphi_1(h\colon fgh \to fg)$, and $k =\varphi_1(q \colon pq \to p) = \varphi_1(h\colon gh \to g)$.
We have the following lemma, whose $n=2,3$ cases are precisely what is encoded in \eqref{OCA:fibercomp}.

\begin{lemma}\label{lem a5 general}
If $\oc$ is a unary operadic category and $\varphi \colon \udec(\oc)\to \oc$ is the fiber functor, then the following diagram commutes for all $n\geq 2$.
\[ \begin{tikzcd}[row sep=tiny]
\oc_n \rar{d_n} \ar[dd,"="]& \oc_{n-1} = \udec(\oc)_{n-2} \ar[dr, start anchor=south east, "\varphi_{n-2}"] \\
& & \oc_{n-2} \\
\udec(\oc)_{n-1} \rar{\varphi_{n-1}} & \oc_{n-1} = \udec(\oc)_{n-2} \ar[ur, start anchor=north east, "\varphi_{n-2}"']
\end{tikzcd} \]
\end{lemma}
\begin{proof}
Let $\sigma$ be an $n$-simplex of $\oc$ as follows
\[
\begin{tikzcd}[sep=small]
x_0 \arrow[r, "f_1"] & x_1 \rar{f_2} & \cdots \rar{f_{n-1}} & x_{n-1} \rar{f_n} & x_n.
\end{tikzcd}
\]
For $i\leq j$ we will use the following shorthand for composites:
\[ f_{i,j} = f_j \circ \cdots \circ f_{i+1} \colon x_i \to x_j. \]
(and likewise for $p_{i,j}$).
Taking fibers twice, we obtain
\[
\begin{tikzcd}[sep=small] 
z_0 \arrow[r, "q_1"]  \dar[no head, densely dotted]& z_1 \rar{q_2}  \dar[no head, densely dotted] & \cdots \rar{q_{n-3}} & z_{n-3} \dar[no head, densely dotted] \rar{q_{n-2}} & z_{n-2} \dar[no head, densely dotted]
\\
y_0 \arrow[r, "p_1"]  \dar[no head, densely dotted]& y_1 \rar{p_2}  \dar[no head, densely dotted] & \cdots \rar{p_{n-3}} & y_{n-3} \dar[no head, densely dotted] \rar{p_{n-2}} & y_{n-2} \dar[no head, densely dotted] \rar{p_{n-1}} & y_{n-1} \dar[no head, densely dotted]
\\
x_0 \arrow[r, "f_1"] & x_1 \rar{f_2} & \cdots \rar{f_{n-3}} & x_{n-3} \rar{f_{n-2}} & x_{n-2} \rar["f_{n-1}"] & x_{n-1} \rar["f_n"] & x_n 
\end{tikzcd}
\]
where 
\[ p_i = \varphi_1 \left( \begin{tikzcd}[column sep=0, row sep=small]
x_{i-1} \ar[rr,"f_i"] \ar[dr,"f_{i-1,n}"'] &  & x_i  \ar[dl,"f_{i,n}"] \\
& x_n
\end{tikzcd} \right)
\qquad
q_j = 
\varphi_1 \left( 
\begin{tikzcd}[column sep=0, row sep=small]
y_{j-1} \ar[rr,"p_j"] \ar[dr,"p_{j-1,n-1}"'] &  & y_j  \ar[dl,"p_{j,n-1}"] \\
& y_{n-1}
\end{tikzcd}
\right)
\]
for $i=1,\dots, n-1$  and $j=1,\dots, n-2$,  

By functoriality of $\varphi$, we have $p_{i,n-1} = p_{n-1} \cdots p_{i+1} = \varphi_1(f_{i,n-1} \colon f_{i,n} \to f_n)$.
Using this, we apply \eqref{OCA:fibercomp} to the following composable pair of morphisms of $\udec(\oc)$
\[ \begin{tikzcd}
x_{i-1} \rar["f_i"] \ar[dr,"f_{i-1,n}"'] & x_i \rar["f_{i,n-1}"] \dar & x_{n-1} \ar[dl,"f_n"] \\
& x_n
\end{tikzcd} \]
to see that 
\[
\varphi_1 \left( \begin{tikzcd}[column sep=0, row sep=small]
x_{i-1} \ar[rr,"f_i"] \ar[dr,"f_{i-1,n-1}"'] &  & x_i  \ar[dl,"f_{i,n-1}"] \\
& x_n
\end{tikzcd} \right) = \varphi_1 \left( 
\begin{tikzcd}[column sep=0, row sep=small]
y_{i-1} \ar[rr,"p_i"] \ar[dr,"p_{i-1,n-1}"'] &  & y_i  \ar[dl,"p_{i,n-1}"] \\
& y_{n-1}
\end{tikzcd}
\right) = q_i.
\]
Thus the result of applying $\varphi_{n-2}$ to the bottom row of the following diagram yields the top row.
\[
\begin{tikzcd}[sep=small] 
z_0 \arrow[r, "q_1"]  \dar[no head, densely dotted]& z_1 \rar{q_2}  \dar[no head, densely dotted] & \cdots \rar{q_{n-3}} & z_{n-3} \dar[no head, densely dotted] \rar{q_{n-2}} & z_{n-2} \dar[no head, densely dotted]
\\
x_0 \arrow[r, "f_1"] & x_1 \rar{f_2} & \cdots \rar{f_{n-3}} & x_{n-3} \rar{f_{n-2}} & x_{n-2} \rar["f_{n-1}"] & x_{n-1}
\end{tikzcd}
\]
This bottom row is $d_n \sigma$, and we conclude that $\varphi_{n-2} \varphi_{n-1} \sigma = \varphi_{n-2} d_n \sigma$.
\end{proof}

\begin{definition}\label{def wbu}
A unary operadic category $\oc$ satisfies the \emph{(weak) blow-up axiom} if the fiber functor
  $\varphi \colon \udec(\oc) \to \oc$ is a discrete opfibration.
\end{definition}

\begin{remark}\label{remark weak strong blow-up}
The weak blow-up axiom was originally given in a different form for general operadic categories in \cite{BataninMarkl:OCNEKD}.
That our definition agrees with that of Batanin--Markl for unary operadic categories follows from \cite[Remark 2.5]{BataninMarkl:OCNEKD}, since in a unary operadic category every morphism is automatically `order-preserving' in their sense.
They also give a stronger blow-up axiom, but \cite[Lemma 2.11]{BataninMarkl:OCNEKD} (combined with \cite[Corollary 2.6]{BataninMarkl:OCNEKD}) imply that for unary operadic categories, the two notions coincide.
Henceforth we omit the adjective `weak.'
\end{remark}

Additionally, we have the following proposition, which follows from Corollary 2.6 and Lemma 2.12 of \cite{BataninMarkl:OCNEKD}.

\begin{proposition}\label{isos and loc term fibers}
Suppose $\oc$ is a unary operadic category satisfying the blow-up axiom.
Let $f \colon x \to z$ be a morphism of $\oc$ having fiber $y$.
\begin{itemize}
\item $y$ is a local terminal object if and only if $f$ is an isomorphism.
\item $y$ is a chosen local terminal object if and only if $f$ is an identity. \qed
\end{itemize}
\end{proposition}

\section{Unary operadic categories as upper d\'ecalage}\label{sec unary as UD}
In this section we expand on part of the proof of the following theorem, which originally appears as \cite[Corollary 12]{GarnerKockWeber:OCD}.
This formulation relies on \cref{lem upper 2-segal char}.
\begin{theorem}[Garner--Kock--Weber]\label{gkw theorem}
The upper d\'ecalage functor induces an equivalence between the category of upper 2-Segal sets and that of unary operadic categories.
\end{theorem}

We recall details about the passage from upper 2-Segal sets to unary operadic categories in Appendix~\ref{app upper dec}.
Given a unary operadic category $\oc$, there is an associated `undecking' $U\oc \in \sset$ which is an upper 2-Segal set such that $\udec (U\oc) \cong \oc$.
In this section, we spell out the details of the simplicial set $X=U\oc$;  
see also the related presentation in \cite[\S1--2]{BataninMarkl:OOCBC}.
Our presentation below has the advantage that equality of simplices is easy to test and most simplicial structure maps are evident, but has the disadvantage that top face maps are obscured.

The set $U\oc_n = X_n \subset \oc_n$ is given by chains $\sigma$ of composable morphisms of length $n$ ending in a chosen local terminal object $u$:
\begin{equation}\label{eq_C_n_elt}
\begin{tikzcd}[sep=small]
x_0 \arrow[r, "f_1"] & x_1 \rar{f_2} & \cdots \rar{f_{n-1}} & x_{n-1} \rar{\tau} & u.
\end{tikzcd}
\end{equation}
In particular, $X_0$ is the set of chosen local terminal objects, $X_1$ is in bijection with the objects of $\oc$, the set $X_2$ is in bijection with the morphisms of $\oc$, and so on.

All of the basic simplicial operators of $X$ agree with those of $\oc$, with the exception of the top face maps.
This should be expected, as there is no reason that the restriction to $X_n$ of $d_n \colon \oc_n \to \oc_{n-1}$ should take values in $X_{n-1}$.
The operator $d_\top^X = d_n^X \colon X_n \to X_{n-1}$ is instead given by the following composite, which lands in $X_{n-1}$ by \eqref{OCA:idfiber}.
\[ \begin{tikzcd}
X_n \rar[hook] \ar[rr, bend right, "\cong"] & \oc_n \rar{d_n} & \oc_{n-1} \rar{s_{n-1}} & \oc_n = \udec(\oc)_{n-1} \rar{\varphi_{n-1}} & \oc_{n-1}
\end{tikzcd} \]
More explicitly, this composite takes \eqref{eq_C_n_elt} to the top row of the following diagram, where $y_i$ is the fiber of the composite map $x_i \to x_{n-1}$. 
\begin{equation*}\label{eq_NDC_n-2_elt}
\begin{tikzcd}[column sep=small] 
y_0 \arrow[r, "p_1"]  \dar[no head, densely dotted]& y_1 \rar{p_2}  \dar[no head, densely dotted] & \cdots \rar{p_{n-2}} & y_{n-2} \dar[no head, densely dotted] \rar & u' \dar[no head, densely dotted]
\\
x_0 \arrow[r, "f_1"] \ar[drrr, bend right=15] & x_1 \rar{f_2}  \ar[drr, bend right=15] & \cdots \rar{f_{n-2}} & x_{n-2} \dar \rar["f_{n-1}"] & x_{n-1} \ar[dl,"\id_{x_{n-1}}"] \\
&  &  &  x_{n-1}
\end{tikzcd}
\end{equation*}
The object $u'$ is a chosen local terminal object by \eqref{OCA:idfiber}, so the top row lies in $X_{n-1}$.
Notice this assignment also makes sense for $n=1$; in this case, $d_1^X( x \xrightarrow{\tau} u)$ will be the chosen local terminal object $\varphi_0(\id_x)$.

\begin{proposition}
The object $X$ is a simplicial set.
\end{proposition}
\begin{proof}
The only simplicial identities that are not automatic involve top face maps.
For $0\leq i \leq n-2$ we have the following string of equalities of maps $X_n \to X_{n-2}$:
\[
d_i d_n^X = d_i \varphi_{n-1} s_{n-1} d_n = \varphi_{n-2} d_i s_{n-1} d_n = \varphi_{n-2} s_{n-2} d_{n-1} d_i = d_{n-1}^X d_i.
\]
For $0\leq i \leq n-1$ we have the following string of equalities of maps $X_n \to X_n$:
\[
  d_{n+1}^X s_i = \varphi_n s_n d_{n+1} s_i = \varphi_n s_n s_i d_n = \varphi_n s_i s_{n-1} d_n = s_i \varphi_{n-1} s_{n-1} d_n = s_i d_n^X.
\]
It remains to check $d_{n-1}^X d_n^X = d_{n-1}^X d_{n-1}$ and $d_{n+1}^X s_n = \id_{X^n}$; the second of these appears as \cref{lem top face degen} below.
For the first equality, we can compute:
\begin{align*}
  d_{n-1}^X d_n^X = \varphi_{n-2} s_{n-2} d_{n-1} \varphi_{n-1} s_{n-1} d_n &= \varphi_{n-2} \varphi_{n-1} s_{n-2} d_{n-1} s_{n-1} d_n \\ &= \varphi_{n-2} \varphi_{n-1} s_{n-2} d_n.
\end{align*}
By \cref{lem a5 general}, we know $\varphi_{n-2} \varphi_{n-1} = \varphi_{n-2} d_n$, so the above expression becomes
\[
  \varphi_{n-2} d_n s_{n-2} d_n = \varphi_{n-2} s_{n-2} d_{n-1} d_n = \varphi_{n-2} s_{n-2} d_{n-1} d_{n-1}.
\]
The right-hand side is equal to $d_{n-1}^X d_{n-1}$.
\end{proof}

\begin{lemma}\label{lem top face degen}
The endomorphism $d_{n+1}^X s_n$ of $X_n$ is the identity.
\end{lemma}
\begin{proof}
Letting $\sigma \in X_n$ be the simplex of \eqref{eq_C_n_elt}, we have $s_n\sigma$ is 
\[
\begin{tikzcd}[sep=small]
x_0 \arrow[r, "f_1"] & x_1 \rar{f_2} & \cdots \rar{f_{n-1}} & x_{n-1} \rar{\tau} & u \rar{\id} & u.
\end{tikzcd}
\]
The top face $d_{n+1}^Xs_n\sigma$ is given by the top line in the following.
\[
\begin{tikzcd}[column sep=small] 
y_0 \arrow[r, "p_1"]  \dar[no head, densely dotted]& y_1 \rar{p_2}  \dar[no head, densely dotted] & \cdots \rar{p_{n-1}} & y_{n-1} \dar[no head, densely dotted] \rar & u' \dar[no head, densely dotted]
\\
x_0 \arrow[r, "f_1"] \ar[drrr, bend right=15] & x_1 \rar{f_2}  \ar[drr, bend right=15] & \cdots \rar{f_{n-1}} & x_{n-1} \dar \rar & u \ar[dl,"\id"] \\
&  &  &  u
\end{tikzcd}
\]
Since $u$ is a chosen local terminal, $p_i = f_i$ for all $i$ by \eqref{OCA:uniquemap}, hence $d_{n+1}^Xs_n\sigma = \sigma$.
\end{proof}

Though the inclusions $X_n \subset \oc_n$ do not assemble into a simplicial map $X\to \oc$, they do induce a simplicial map $\udec(X) \hookrightarrow \udec(\oc)$ as the upper d\'ecalage discards top face maps.

\begin{proposition}
If $\oc$ is a unary operadic category and $X = U\oc$ is the simplicial set constructed above, then $\udec(X) \cong \oc$.
\end{proposition}
\begin{proof}
The composite of $\udec(X) \hookrightarrow \udec(\oc)$ with the canonical augmentation $d_\top \colon \udec(\oc) \to \oc$ is an isomorphism. 
By inspection, the map \[ \udec(X)_n = X_{n+1} \subset \oc_{n+1} \xrightarrow{d_{n+1}} \oc_n \] is a bijection, whose inverse is given by adjoining the map to a chosen local terminal object.
\end{proof}

As a consequence, $X$ is upper 2-Segal.
This simplicial set is necessarily isomorphic to the one given in \cite{GarnerKockWeber:OCD}.
The reader is encouraged to compare our description of $X$ with the low-dimensional pictures from \cite[page 13]{GarnerKockWeber:OCD}.
Generally, the full data of an $n$-simplex of $X$ is a grid consisting of objects and morphisms of $\oc$, such that each row is the fiber of the one below.
\[
\begin{tikzcd}
x_0^{n-1} \dar[no head, densely dotted] 
\\
x_0^{n-2} \arrow[r, "f_1^{n-2}"]  \dar[no head, densely dotted]& x_1^{n-2}   \dar[no head, densely dotted] 
\\
\vdots \dar[no head, densely dotted]& \vdots   \dar[no head, densely dotted] & \cdots & \vdots  \dar[no head, densely dotted]  
\\
x_0^1 \arrow[r, "f_1^1"]  \dar[no head, densely dotted]& x_1^1 \rar{f_2^1}  \dar[no head, densely dotted] & \cdots \rar{f_{n-3}^1} & x_{n-3}^1 \dar[no head, densely dotted] \rar{f_{n-2}^1} & x_{n-2}^1 \dar[no head, densely dotted] 
\\
x_0^0 \arrow[r, "f_1^0"] & x_1^0 \rar{f_2^0} & \cdots \rar{f_{n-3}^0} & x_{n-3}^0 \rar{f_{n-2}^0} & x_{n-2}^0 \rar["f_{n-1}^0"] & x_{n-1}^0 
\end{tikzcd}
\]
Identifying the 1-simplices of $X$ with objects of $\oc$, the spine of this $n$-simplex is
\[ \begin{tikzcd}
u^n \rar{x_0^{n-1}} & u^{n-1} \rar{x_1^{n-2}} & \cdots \rar{x_{n-2}^1} & u^1 \rar{x_{n-1}^0} & u^0 
\end{tikzcd} \]
where $u^i$ is the chosen local terminal in the $i$th row and/or the fiber of $\id_{x_{n-i}^{i-1}}$.

\section{Fiber inclusions via lower d\'ecalage}\label{sec fiber inclusions}

The following counterpart to \cref{gkw theorem} appears in \cite[Remark 13]{GarnerKockWeber:OCD}.
\begin{theorem}[Garner--Kock--Weber]\label{thm udec}
The upper d\'ecalage functor induces an equivalence between the category of 2-Segal sets and that of unary operadic categories satisfying the blow-up axiom.
\end{theorem}
\begin{proof}
Suppose $X$ is an upper 2-Segal set and $\oc = \udec (X)$ is the associated unary operadic category.
For $n\geq 1$, the three squares below are equal.
\[ \begin{tikzcd}
\udec(\oc)_n \rar{d_\top} \dar[swap]{\varphi_n} & \udec(\oc)_{n-1} \dar{\varphi_{n-1} }
& 
\oc_{n+1} \rar{d_n} \dar[swap]{\varphi_n} & \oc_n \dar{\varphi_{n-1} }
&
X_{n+2} \rar{d_n} \dar[swap]{d_{n+2}} & X_{n+1} \dar{d_{n+1}}
\\
\oc_n \rar[swap]{d_\top} & \oc_{n-1} 
&  
\oc_n \rar[swap]{d_n} & \oc_{n-1} 
& 
X_{n+1} \rar[swap]{d_n} & X_n 
\end{tikzcd} \]
The left square is a pullback for all $n\geq 1$ if and only if $\varphi$ is a discrete opfibration. 
By \cite[Lemma 3.6]{GKT1}, the square on the right is a pullback for all $n\geq 1$ if and only if $X$ is lower 2-Segal.
Applying \cref{gkw theorem}, the result follows.
\end{proof}

Suppose $\oc$ is a unary operadic category, which moreover satisfies blow-up axiom.
By \cref{thm udec}, the `undecking' $U\oc = X$ is a 2-Segal set.
Hence the lower d\'ecalage $\ldec(X)$ is also a category by \cref{lem upper 2-segal char}.
In this section we interpret $\ldec(X)$ as a category of formal fiber inclusions of $\oc$.
Below we will use $d^\bot_i, s^\bot_j$ to indicate the simplicial operators on $\ldec(X)$.

\begin{definition}
If $\oc$ is a unary operadic category satisfying the blow-up axiom, then the \emph{category of fiber inclusions}, $\fc$, has the same objects as $\oc$ and the set of maps from $x$ to $y$ in $\fc$ coincide with the set of maps in $\oc$ with domain $y$ and fiber $x$.
We write $[g] \colon x \dashrightarrow y$ for the morphism in $\fc$ associated with $g \colon y \to z$ having fiber $x$.
More formally, $\fc \cong \ldec(U\oc)$.
\end{definition}

That is, the same element $\sigma = (y \xrightarrow{g} z \to u)$ in $X_2$ is interpreted as $g  \colon y \to z$ in $\oc_1 \cong \udec(X)_1$, and as $[g] \colon \varphi_0(g) \dashrightarrow y$ in $\fc_1 \cong \ldec(X)_1$.
Indeed, $d_1^\bot \sigma = d_2^X \sigma = (\varphi_0(g) \to u')$ and $d_0^\bot \sigma = d_1\sigma = (y\to u')$, so $\sigma$ represents a map with the indicated domain and codomain.

\begin{proposition}\label{prop idens in fc}
Let $\oc$ be a unary operadic category satisfying the blow-up axiom, and $\fc$ its category of fiber inclusions.
The map $\tau_x \colon x \to u$ in $\oc$ to a chosen local terminal object induces the identity map $\id^\fc_x = [\tau_x] \colon x \dashrightarrow x$ in $\fc$.
\end{proposition}
\begin{proof}
The object $x$ is represented by $\tau_x \colon x\to u$ in $\ldec(X)_0 = X_1$.
The identity is given by $s_0^\bot(x\to u) = s_1(x\to u) = (x\to u \to u)$, which corresponds to $[\tau_x] \colon x \dashrightarrow x$ in $\fc$.
\end{proof}

\begin{proposition}\label{prop fiber composition}
Let $\oc$ be a unary operadic category satisfying the blow-up axiom, and $\fc$ its category of fiber inclusions.
Suppose we have a composable pair of maps 
\[ \begin{tikzcd}[sep=small]
x \rar[dashed]{[f]} & y \rar[dashed]{[g]} & z
\end{tikzcd} \]
in $\fc$.
That is, $f \colon y \to a$ in $\oc$ has fiber $x$, and $g\colon z \to b$ in $\oc$ has fiber $y$.
Then by the weak blow-up axiom there is a unique factorization $g = hk$ as in the following diagram so that $\varphi_1(k) = f$. 
The composite in $\fc$ is given by $[g] \circ [f] = [k] \colon x \dashrightarrow z$.
\[ \begin{tikzcd}[row sep=small, column sep=tiny]
y \ar[rr,"f"] \dar[no head, densely dotted] & & a \dar[no head, densely dotted]\\
z \ar[rr,"k"] \ar[dr,"g"'] & & c \ar[dl,"h"]  \\
& b
\end{tikzcd} \]
\end{proposition}
\begin{proof}
Since $X$ is lower 2-Segal, the following square is a pullback
\[ \begin{tikzcd}
X_3 \rar{d_1} \dar[swap]{d_3^X} \ar[dr, phantom, "\lrcorner" very near start] & X_2 \dar{d_2^X}  \\
X_2 \rar[swap]{d_1} & X_1 
\end{tikzcd} \] 
so there is a unique element $\sigma \in X_3$ of the form
\[ \begin{tikzcd}[sep=small]
z \rar{k} & c \rar{h} & b \rar & u
\end{tikzcd} \]
such that $d_0^\bot \sigma = d_1\sigma = (z \xrightarrow{g} b \to u)$ and $d_2^\bot \sigma =d_3^X\sigma = (y\xrightarrow{f} a \to u')$.
Hence $\sigma \in \ldec(X)_2$ represents the composable pair $([g], [f])$.
The composite is given by
\[
  d_1^\bot \sigma = d_2 \sigma = (\begin{tikzcd}[sep=small,cramped]
z \rar{k} & c \rar & u
\end{tikzcd}).
\]
This element corresponds to $[k] \colon x \dashrightarrow z$ in $\fc_1$.
\end{proof}

\section{Isomorphisms in the category of fiber inclusions}\label{sec fiber isos}

In this section, we classify the isomorphisms in $\fc$.
These turn out to be the same as the category of \emph{virtual isomorphisms} of Batanin and Markl, at least when virtual isomorphisms are defined. These virtual isomorphisms are used in constructing indexing categories for Markl operads over (not necessarily unary) operadic categories.

\begin{lemma}\label{lem right invertible}
Let $\oc$ be a unary operadic category satisfying the blow-up axiom and $\fc$ its category of fiber inclusions.
Suppose $g \colon x \to z$ is a morphism of $\oc$ with fiber $y$.
Then $[g] \colon y \dashrightarrow x$ in $\fc$ admits a right inverse if and only if $g$ factors through a chosen local terminal object.
\end{lemma}
\begin{proof}
A composable pair of morphisms
\[ \begin{tikzcd}[column sep=small]
x \rar["\tau_x"] & u \rar{k} & z \rar["\tau_z"] & u,
\end{tikzcd} \]
denoted by $\sigma$ in $\ldec(X)_2 = X_3$, has that $d_2^\bot\sigma = d_3^X\sigma$ is the top line in
\begin{equation}\label{eq construction of inverse}
\begin{tikzcd}[sep=small]
y \dar[no head, densely dotted] \rar["h"] & w \rar \dar[no head, densely dotted] & u' \dar[no head, densely dotted]\\
x \rar["\tau_x"] & u \rar["k"] & z \rar["\id_z"] & z.
\end{tikzcd} \end{equation}
We also have $d_0^\bot\sigma = d_1 \sigma = (x \to z \to u)$ and $d_1^\bot\sigma = d_2\sigma = (x \to u \to u) = s_1\tau_x = s_0^\bot (\tau_x)$.
It follows that $[k\tau_x] \circ [h]$ is an identity.
Hence if $g \colon x \to z$ factors through a chosen local terminal object, then $[g]$ admits a right inverse.
On the other hand, if $[g]$ admits a right inverse then there is $\sigma \in \ldec(X)_2 = X_3$ witnessing this fact; write this element $\sigma$ as 
\[ \begin{tikzcd}[column sep=small]
x \rar["\ell"] & a \rar{k} & z \rar & u.
\end{tikzcd} \]
Since $d_0^\top \sigma = d_1\sigma = (x \to z \to u)$, we have $k\ell = g$ in $\oc$.
But $d_1^\top \sigma = d_2\sigma = (x\to a \to u)$ is assumed to be $(x \to u \to u)$, so we must have $a=u$.
Thus $g$ factors through a chosen local terminal.
\end{proof}

\begin{theorem}\label{thm fiber isos}
Let $\oc$ be a unary operadic category satisfying the blow-up axiom and $\fc$ its category of fiber inclusions.
Suppose $g \colon x \to z$ is a morphism of $\oc$ with fiber $y$.
Then $[g] \colon y \dashrightarrow x$ in $\fc$ is invertible if and only if $z$ is a local terminal object of $\oc$.
\end{theorem}
\begin{proof}
In either case, $g$ factors as $k\tau_x \colon x \to u \to z$.

Suppose $z$ is a local terminal object.
Then $k\colon k \to \id_z$ is an isomorphism in $\udec(\oc)$, so the morphism $(w\to u') = \varphi_1(k)$ in \eqref{eq construction of inverse} (from the proof of \cref{lem right invertible}) is also an isomorphism.
Now $[h]$ is a right inverse for $[g]$, and has its own right inverse by \cref{lem right invertible}, hence $[g]$ is invertible.

For the reverse direction, suppose that $[g]$ is invertible, and let $h \colon y \to w$ be as in the proof of \cref{lem right invertible}, so that $[h]$ is the inverse of $[g]$.
That is, $h = \varphi_1(\tau_x \colon g \to k)$.
Since $[h]$ is invertible, $h$ factors as $j\tau_y \colon y \to u' \to w$.
Likewise, we have $g = \varphi_1(\tau_y \colon h \to j)$.
Since $\varphi$ is a discrete opfibration, there is a unique element $\sigma' \in \oc_3 = \udec(\oc)_2$
\[
\begin{tikzcd}[column sep=small]
  x \rar{g} & z \rar{m} & a \rar{n} & z
\end{tikzcd}
\]
with $d_2 \sigma' = s_1g$ and 
\[
  \varphi_2(\sigma') = (\begin{tikzcd}[sep=small, cramped]
    y \rar{\tau_y} & u' \rar{j} & w
    \end{tikzcd}).
\]
This assertion uses that $\varphi_0(g) = y$, along with the fact that $\varphi_1 ( g \colon g \to \id_z)$ is a map  $y = \varphi_0(g) \to \varphi_0(\id_z)$, so must be $\tau_y$ by \eqref{OCA:idfiber}.

Now consider the corresponding element 
\[
\sigma = \left( 
\begin{tikzcd}[column sep=small, cramped]
  x \rar{g} & z \rar{m} & a \rar{n} & z \rar{\tau} & u
\end{tikzcd} \right)
\]
in $X_4 = \ldec(X)_3$.
In a moment, we will verify that this represents the chain
\begin{equation}\label{eq chain}
\begin{tikzcd}
y \rar[dashed, "{[g]}"] & x \rar[dashed, "{[h]}"]  & y \rar[dashed, "{[g]}"] & x,
\end{tikzcd} \end{equation}
in $\fc$, which implies that 
\[d_2^\bot \sigma = s_1^\bot [g] = s_2 (x \xrightarrow{g} z \to u) = (x \xrightarrow{g} z \to u \to u ).\]
On the other hand, $d_2^\bot \sigma = d_3 \sigma$ is computed directly as 
\[
  \begin{tikzcd}[column sep=small, cramped]
  x \rar{g} & z \rar{m} & a \rar & u
\end{tikzcd}
\]
which implies that $a = u$. 
Since $\id_z = nm = n\tau_z$, this implies that $z \cong u$.
This is precisely what we are trying to prove, that $z$ is a local terminal object.

Let us verify that $\sigma$ actually represents the chain \eqref{eq chain}, which will complete the proof.
We have $d_3^\bot \sigma = d_4^X\sigma$ is 
\[
   \varphi_3\left( 
\begin{tikzcd}[column sep=small, cramped]
  x \rar{g} & z \rar{m} & a \rar{n} & z \rar{\id_z} & z 
\end{tikzcd} \right)
= 
\left( 
\begin{tikzcd}[column sep=small, cramped]
  y \rar{\tau_y} & u' \rar{j} & w \rar{\tau} & u'  
\end{tikzcd} \right)
 \] 
which is the element corresponding to $\varphi_2(\sigma')$.
We then have $d_2^\bot d_3^\bot \sigma = d_3^X d_4^X \sigma$ is
\[
   \varphi_2\left( 
\begin{tikzcd}[column sep=small, cramped]
  y \rar{\tau_y} & u' \rar{j} & w \rar{\id_w} & w  
\end{tikzcd} \right)
= 
\left( 
\begin{tikzcd}[column sep=small, cramped]
  x \rar{g} & z \rar & u 
\end{tikzcd} \right) = [g]
 \] 
On the other hand, $d_0^\bot d_3^\bot \sigma = d_1 d_3^\bot \sigma = (y \xrightarrow{h} w \to u')$, which is just $[h]$.
Finally, $d_0^\bot d_0^\bot \sigma = d_1 d_1 \sigma$ is just $[g]$ since $nm = \id_z$. 
We conclude that $\sigma$ is really of the form \eqref{eq chain}, so $z$ is a local terminal object.
\end{proof}

With several assumptions on an operadic category, in \cite[\S2]{BataninMarkl:KDOC} the notion of `virtual morphism' is defined, which are precisely formal fiber inclusions associated to maps with codomain a local terminal object.
The \emph{unique fiber condition} \cite[Definition 2.14]{BataninMarkl:OCNEKD} says that if $t$ is a local terminal object and $x\to t$ has fiber $x$, then $t$ is a chosen local terminal object.

\begin{corollary}\label{cor virtual isos}
Suppose $\oc$ is a unary operadic category satisfying the blow-up axiom and the unique fiber condition.
Then the category of virtual isomorphisms is the subcategory of isomorphisms in $\fc$.
\end{corollary}
\begin{proof}
Virtual isomorphisms are defined in \cite[\S2]{BataninMarkl:KDOC} only under certain hypotheses, which follow from our assumptions on $\oc$ by \cite[Lemma 2.11]{BataninMarkl:OCNEKD} (see also \cref{remark weak strong blow-up}).
Combining \cref{thm fiber isos} with the unique fiber condition implies, by \cite[Lemma 2.15]{BataninMarkl:OCNEKD}, that the maximal subgroupoid of $\fc$ is a preorder.
The same is true (for the same reasons) of the category of virtual isomorphisms as explained in \cite[Lemma 2.1]{BataninMarkl:KDOC}, so these two groupoids coincide.
\end{proof}

\section{The expanded category}\label{sec expanded}

\begin{definition}
Suppose $\oc$ is a unary operadic category satisfying the blow-up axiom.
Then there is a category $\expandC$ with the same objects as $\oc$ and morphisms are pairs of the form
\[ \begin{tikzcd}[sep=small]
x \rar["\lr{f}", squiggly]  & y \rar{g} & z
\end{tikzcd} \]
where $f$ in $\oc$ is a morphism with domain $x$ and fiber $y$ and $g \colon y \to z$ is in $\oc$. 
Formally, the expanded category $\expandC$ is isomorphic to $\sd(U\oc)^\op$ and this representation of morphisms utilizes the canonical factorization system from \cref{prop canonical fact}.
\end{definition}

Indeed, if $X = U\oc$, then $(\iudecop(X), \ildec(X))$ is a strict factorization system on $\sd (X)$, so $(\ildec(X)^\op, \iudecop(X)^\op)$ is a strict factorization system on $\sd(X)^\op$.
Here, $\iudecop(X)^\op \cong (\udec(X)^\op)^\op \cong \oc$, while $\ildec(X)^\op \cong \ldec(X)^\op \cong \fc^\op$. 
Our representation for morphisms utilizes the strict factorization system $(\fc^\op, \oc)$ on $\expandC$.

We now explicitly spell out the distributive law governing the multiplication in terms of the operadic category structure on $\oc$ \cite{Beck:DL,RosebrughWood:DLF}.

\begin{proposition}\label{prop distrib law}
Suppose we have a composable pair of morphisms
\begin{equation*} \begin{tikzcd}[sep=small]
x \rar{f} & y \rar["\lr{g}", squiggly] & z
\end{tikzcd} \end{equation*}
in $\expandC$, with $f\in \oc$ and $\lr{g} \in \fc^\op$.
If $h = \varphi_1 (f \colon gf \to g)$, then the composite $\lr{g} \circ f$ is
\[
\begin{tikzcd}[sep=small]
x \rar["\lr{gf}", squiggly] & \varphi_0(gf) \rar["h"] & z.
\end{tikzcd}
\]
\end{proposition}
\begin{proof}
In $\fc$ we have $[g] \colon z \dashrightarrow y$, which implies the domain of $g$ is $y$.
We then have the following element $\sigma \in X_3 = \sd(X)_1$
\[ \begin{tikzcd}[sep=small]
x \rar{f} & y \rar{g} & w \rar & u.
\end{tikzcd} \]
We first claim that this morphism in $\sd(X)$ represents (the opposite of) $\lr{g} \circ f$ in $\expandC$.
To see this, note the factorization of $\sigma$ provided by \cref{second sfs} is given by $s_1s_3\sigma \in X_5$, whose constituent maps can be computed using the following identities
\[ 
\begin{gathered}
d_2^\sdrm s_1s_3  = d_0d_5 s_1 s_3 = d_0s_1d_4s_3 = s_0d_0  \\
d_0^\sdrm s_1s_3  = d_2d_3 s_1 s_3  = d_2 s_1 d_2 s_3 = s_2d_2
\end{gathered}
\]
in $X$. 
Thus the first map of $s_1s_3\sigma \in \sd(X)_2$ is $s_0d_0 \sigma = (y \to y \to w \to u)$ which corresponds to $[g] \colon z \dasharrow y$ under $\ildec(X) \cong \fc$, while the second map of $s_1s_3\sigma \in \sd(X)_2$ is $s_2d_2 \sigma = (x \to y \to u \to u)$ which corresponds to $f^\op \colon y \to x$ under the isomorphism $\iudecop(X) \cong \oc^\op$.

We now compute the factorization of $\sigma \in \sd(X)_1$ provided by \cref{prop canonical fact}.
This is given by $s_0s_2\sigma \in X_5 = \sd(X)_2$.
We have the following basic identities, emphasizing that the top face maps appearing in the first line are all $d_\top^X$.
\[ 
\begin{gathered}
d_2^\sdrm s_0s_2  = d_0d_5 s_0 s_2 = d_0s_0d_4s_2 =  s_2 d_3 \\
d_0^\sdrm s_0s_2  = d_2d_3 s_0 s_2  = d_2 s_0 d_2 s_2 = s_0 d_1
\end{gathered}
\]
Now $d_3^X \sigma$ is the top line in the following:
\[
\begin{tikzcd}[sep=small]
m \dar[no head, densely dotted] \rar["h"] & z \rar \dar[no head, densely dotted] & u' \dar[no head, densely dotted]\\
x \rar["f"] & y \rar["g"] & w \rar["\id_w"] & w,
\end{tikzcd} \]
so we see that the first map $s_2 d_3^X\sigma$ corresponds to $h^\op \colon z \to m$ under the isomorphism $\iudecop(X) \cong \oc^\op$.
The second morphism is $s_0d_1\sigma = (x\to x \to w \to u)$ which corresponds to $[gf] \colon m \dashrightarrow x$ under $\ildec(X) \cong \fc$.
\end{proof}

It is of course possible to establish the existence of the expanded category $\expandC$ without reference to 2-Segal sets.
One takes the statement of \cref{prop fiber composition} as the definition of composition in the category of fiber inclusions, shows that this forms a category (with $[\tau_x]$ the identities), and proves directly that the rule from \cref{prop distrib law} defines a distributive law between $\oc$ and $\fc^\op$.
Then $\expandC$ is formed by combining these categories \cite[2.6]{RosebrughWood:DLF}.
This was how we originally came across $\expandC$ (and also \cref{prop canonical fact}), when we were applying this strategy to not necessarily unary operadic categories.
The correct definition of the category of fiber inclusions is still unclear in the non-unary case, but we would like to explore this further.

\begin{remark}
Suppose $\oc$ is a unary operadic category satisfying the blow-up axiom and the unique fiber condition, so that isomorphisms in $\fc$ coincide coincide with virtual isomorphisms (\cref{cor virtual isos}).
Then the category of isomorphisms in $\expandC$ coincides with the category called $\oc_\textup{vrt}\int\oc_\textup{iso}$ in \cite{BataninMarkl:KDOC}.
This is simply because the category of isomorphisms in $\expandC$ are those formal composites $g \circ \lr{f}$ such that both $g$ and $\lr{f}$ are isomorphisms, by \cite[Proposition 2.13]{RosebrughWood:DLF}.
\end{remark}

\section{Hypermoment structure}\label{sec hypermoment}
Our initial motivation was to investigate the reverse direction of \cref{berger theorem}.
So far we have only discussed (strict) factorization systems, and have not addressed the issue of whether or not $\sd(U\oc)$ constitutes a hypermoment category.
A hypermoment category consists of an orthogonal factorization system and a cardinality functor to Segal's category $\Gamma$ (the opposite of finite pointed sets), satisfying several conditions.
We recall the following definition from \cite[\S3]{Berger:MCO}, though we note that in this section we only consider cardinality functors $\gamma$ which are constant at $\underline{1} \in \Gamma$.

\begin{definition}[Hypermoment category]\label{def hypermoment}
The data for a hypermoment category consists of a category $\ec$, an active/inert factorization system $(\ec_\actrm, \ec_\intrm)$ on $\ec$, and a cardinality functor $\gamma$ from $\ec$ to $\Gamma$ compatible with the factorization system.
Given such data, an object $x$ is a \emph{unit} if $\gamma(x)=\underline{1}$ and every active map with codomain $x$ has a unique inert section.
This data determines a \emph{hypermoment category} if for each $y \in \ec$ and each inert map $\underline{1} \rightarrowtail \gamma(y)$ in $\Gamma$, there is an essentially unique inert lift $x \rightarrowtail y$ where $x$ is a unit.
We have $\ec$ is \emph{unital} if every object $y$ admits an essentially unique active map $x' \ract y$ whose domain is a unit.
Finally, $\ec$ is \emph{rigid} if every isomorphism is an automorphism, and automorphisms act trivially (on the left) on active morphisms whose domain is a unit. 
\end{definition}

As observed in \cite[2.1]{Grandis:WSEMCC}, each strict factorization system determines a unique enveloping orthogonal factorization system (see also \cite[3.1]{RosebrughWood:DLF}).
When applied to the strict factorization system $(\oc^\op, \fc)$ on $\ec \coloneqq \expandC^\op \cong \sd(U\oc)$, we call morphisms in the left class \emph{active maps} and morphisms in the right class \emph{inert maps}, and we obtain an orthogonal factorization system $(\ec_\actrm, \ec_\intrm)$ on $\ec$.
The classes of active and inert maps are given by
\begin{align*}
\ec_\actrm &\coloneqq \{ [g] \circ f^\op \mid [g] \text{ is invertible} \} \\
\ec_\intrm &\coloneqq \{ [g] \circ f^\op \mid f^\op \text{ is invertible} \}
\end{align*}
and we write $x \ract y$ for active maps and $x \rightarrowtail y$ for inert maps.

\begin{remark}\label{remark ambidexterity}
The pair $(\ec_\intrm, \ec_\actrm)$ is also an orthogonal factorization system on $\ec$.
This is a general fact about \emph{ambidextrous} strict factorization systems, that is those strict factorization systems $(\mathcal{A}, \mathcal{B})$ such that $(\mathcal{B}, \mathcal{A})$ is also a strict factorization system (equivalently, the associated distributive law is invertible).
In more detail, suppose that $(\mathcal{L}, \mathcal{R})$ is the enveloping orthogonal factorization system for $(\mathcal{A}, \mathcal{B})$ and $(\mathcal{L}', \mathcal{R}')$ that for $(\mathcal{B}, \mathcal{A})$.
As observed in \cite[2.1]{Grandis:WSEMCC} for $(\mathcal{A},\mathcal{B})$, if we have $ab = b'a'$ with $b$ an isomorphism, then $b'$ is an isomorphism as well (here $a,a' \in \mathcal{A}$ and $b,b' \in \mathcal{B}$).
Ambidexterity gives the reverse implication, so $\mathcal{L} = \mathcal{R}'$.
Similarly we have $\mathcal{R} = \mathcal{L}'$, and we conclude $(\mathcal{L}',\mathcal{R}') = (\mathcal{R}, \mathcal{L})$.
\end{remark}

\begin{definition}\label{def unit}
We say that an object $x$ is a \emph{unit} of $\ec$ if every active map with codomain $x$ has precisely one inert section.
\[ \begin{tikzcd}[column sep=small]
x \ar[rr, tail, "\exists !"] \ar[dr,"="'] & & y \ar[dl,-act] \\
& x
\end{tikzcd} \]
\end{definition}

\begin{lemma}\label{lem units imply no maps out}
If $x$ is a unit of $\ec = \expandC^\op \cong \sd(U\oc)$ and $p \colon x \to y$ is any map in $\oc$, then $p$ is an isomorphism.
In particular, any unit is a local terminal object.
\end{lemma}
\begin{proof}
The map $p^\op \colon y \ract x$ is an active map with codomain $x$, hence it has a unique inert section $x \rightarrowtail y$.
We first factor this inert map into a map $f^\op$ in $\oc^\op$ followed by a map $[g]$ in $\fc$, as in the triangle below left, using the strict factorization system $(\oc^\op, \fc)$ from \cref{prop canonical fact}.
\[ \begin{tikzcd}[column sep=small]
x \ar[rr,tail] \ar[dr,"f^\op"',-act]  && u \ar[rr,-act,"p^\op"] && x\\
& y \ar[ur,dashed,"{[g]}"'] \ar[rr,-act,"h^\op"'] & & x \ar[ur,dashed,"{[k]}"']
\end{tikzcd} \]
Notice that $f^\op$ is an isomorphism in $\oc^\op$ since this part of the definition of the class of inert maps.
We then factor $p^\op \circ [g]$ as $[k] \circ h^\op$.
Since the identity on $x$ is $[k] \circ (fh)^\op$ and we are working in a strict factorization system, we have $[k]$ and $(fh)^\op$ are identities. 
Thus $p^\op \circ [g] = [k] \circ h^\op = h^\op$, so using the other strict factorization system $(\fc, \oc^\op)$ from \cref{second sfs}, we conclude that $h = p$ and that $[g]$ is an identity.
We then have $fp = fh = \id_x$.
Since $f$ is invertible, so is $p$.

The second statement follows by taking $p=\tau_x \colon x\to u$.
\end{proof}

\begin{lemma}\label{lem units char}
An object $x$ is a unit of $\ec$ if and only if every map in $\oc$ with domain $x$ is an isomorphism.
\end{lemma}
\begin{proof}
The forward direction was proved in \cref{lem units imply no maps out}.
Suppose that every morphism in $\oc$ with domain $x$ is an isomorphism, and let $w \ract x$ be an active map of $\ec$.
We factor this map as 
\[ \begin{tikzcd}
w \rar[tail, dashed, "{[g]}"] & z  \rar[-act, "f^\op"] & x
\end{tikzcd} \]
using the strict factorization system $(\fc, \oc^\op)$ from \cref{second sfs}.
By assumption, $f\colon x \to z$ is an isomorphism.
Since $(\ec_\intrm,\ec_\actrm)$ is an orthogonal factorization system by \cref{remark ambidexterity}, we also know $[g]$ is an isomorphism.
Hence $w\ract x$ is an isomorphism, which implies it has a unique inert section. 
We conclude that $x$ is a unit.
\end{proof}

\begin{remark}
The proofs of \cref{lem units imply no maps out} (excluding the statement about local terminals) and \cref{lem units char} are general and apply to any category $\ec$ endowed with an ambidextrous strict factorization system $(\mathcal{A}, \mathcal{B})$.
That is, if $(\mathcal{L}, \mathcal{R})$ is the enveloping orthogonal factorization system, then the units of $(\ec, (\mathcal{L}, \mathcal{R}))$ are precisely those $x$ such that every map of $\mathcal{A}$ with codomain $x$ is an isomorphism.
\end{remark}

As we are trying to build rigid hypermoment categories (see \cref{berger theorem}), the following lemmas are relevant.

\begin{lemma}\label{lem isos are idens}
Let $\oc$ be a unary operadic category satisfying the blow-up axiom.
If every isomorphism in $\oc$ is an automorphism, then every isomorphism in $\oc$ is an identity.
\end{lemma}
\begin{proof}
Suppose that $g\colon c \to c$ is an automorphism of $\oc$.
Then the fiber of $g$ is a local terminal object by \cref{isos and loc term fibers}. 
Alternatively, we have $\varphi_1(g) = (w\to u)$ below is an isomorphism, hence an identity.
\[ \begin{tikzcd}[row sep=small, column sep=tiny]
w \ar[rr] \dar[no head, densely dotted] & & u \dar[no head, densely dotted]\\
c \ar[rr,"g"] \ar[dr,"g"'] & & c \ar[dl,"\id_c"]  \\
& c
\end{tikzcd} \]
Since the fiber of $g$ is $u$, \cref{isos and loc term fibers} implies that $g$ is an identity.
\end{proof}

\begin{lemma}\label{lem isos idens fc}
Let $\oc$ be a unary operadic category satisfying the blow-up axiom.
If every isomorphism in $\oc$ is an automorphism, then every isomorphism in the category of fiber inclusions $\fc$ is an identity.
\end{lemma}
\begin{proof}
Suppose $[g] \colon y \dashrightarrow x$ is invertible.
By \cref{thm fiber isos}, the codomain of $g \colon x \to z$ is a local terminal object.
The previous lemma implies that $z$ is a chosen local terminal object, so $g = \tau_x$.
\Cref{prop idens in fc} implies that $[g]=[\tau_x]$ is an identity.
\end{proof}

In this special situation, we have $(\ec_\actrm, \ec_\intrm) = (\oc^\op, \fc)$ and every isomorphism of $\ec$ is an identity.

\begin{proposition}\label{prop reverse berger}
Suppose $\oc$ is a unary operadic category satisfying the blow-up axiom, such that every isomorphism is an automorphism.
If the only morphisms in $\oc$ with domain a (chosen) local terminal object are identities, then $\ec = \expandC^\op$ is a rigid unital hypermoment category.
Further, applying \cref{berger theorem} to $\ec$ recovers $\oc$.
\end{proposition}
\begin{proof}
As mentioned at the start of this section, we take the cardinality functor $\ec \to \Gamma$ to be identically $\underline{1}$.
Thus the units from \cref{def hypermoment} and \cref{def unit} coincide.
By our assumption and \cref{lem units char}, the units are precisely the local terminal objects.
Since there is only one inert endomorphism of $\underline{1}$ in $\Gamma$, to prove that $\ec$ is a hypermoment category we must show that for each $x\in \ec$, there is an essentially unique inert map $w \rightarrowtail x$ with $w$ a unit.
\Cref{lem isos are idens} implies that the category of inert maps $\ec_\intrm$ is equal to the category of fiber inclusions $\fc$, and that every local terminal is a chosen local terminal.
By \cref{isos and loc term fibers}, $[\id_x]$ is the unique map of $\fc$ whose domain is a unit and whose codomain is $x$.
Thus $\ec$ is a hypermoment category.

The unitality axiom for a hypermoment category says that each object $x$ admits an essentially unique active map from a unit.
Since $\ec_\actrm = \oc^\op$ by \cref{lem isos idens fc} and units are the chosen local terminals, it is clear there is a unique map $\tau_x^\op \colon u \ract x$.
Rigidity is automatic since every isomorphism in $\ec$ is an identity.
The operadic category associated to $\ec$ is $\ec_\actrm^\op = \oc$.
\end{proof}

\ifhs
\begin{appendices}
\else
\appendix
\fi

\section{Upper d\'ecalage of an upper 2-Segal set}\label{app upper dec}

Given an upper 2-Segal set $Z$, the simplicial set $\oc = \udec(Z)$ is a category by \cref{lem upper 2-segal char}.
It comes equipped with a canonical fiber map \[ 
\begin{gathered}  
\varphi \colon \udec (\oc) = \udec\udec (Z) \to \udec (Z) = \oc \\
Z_{n+2} \ni \sigma \xmapsto{\varphi_n} d_{n+2}\sigma \in Z_{n+1}. 
\end{gathered}
\]
Meanwhile, $Z_0 \xrightarrow{s_0} Z_1 = \oc_0$ endows $\oc$ with chosen local terminal objects. 
Together, these give $\oc$ the structure of an operadic category.

\begin{lemma}
Suppose $Z$ is upper 2-Segal and $\oc = \udec(Z)$. 
The elements of $s_0(Z_0) \subset Z_1 = \oc_0$ are a set of local terminal objects of $\oc$.
More precisely, each object of $\oc$ admits exactly one morphism to an object of $s_0(Z_0)$.
\end{lemma}
\begin{proof}
In an upper 2-Segal set, if $\sigma \in Z_2$ is such that $d_0\sigma = s_0d_0d_1\sigma$ then $\sigma = s_1 d_1 \sigma$.
Indeed, if the first equality holds, then $s_1 \sigma = s_2 \sigma \in Z_3 = \oc_2$ since $\oc$ is a category; we simply must show that $d_0$ and $d_2$ of these elements agree. 
Of course $d_2$ of both of these is just $\sigma$, and for $d_0$ we have
\[ d_0 s_1 \sigma = s_0 d_0 \sigma = s_0 s_0 d_0d_1 \sigma = s_1 s_0 d_0 d_1 \sigma = s_1 d_0 \sigma = d_0 s_2 \sigma. \]
Using $s_1 \sigma = s_2\sigma$, we compute $\sigma = d_1 s_1 \sigma = d_1 s_2 \sigma = s_1 d_1 \sigma$, as desired.

Now suppose that $f\in Z_1 = \oc_0$ is an arbitrary object of $\oc$.
If $\sigma \in Z_2$ is any element with $d_1 \sigma = f$ and $d_0 \sigma = s_0z \in s_0(Z_0)$, then 
\[
  d_0 f = d_0 d_1 \sigma = d_0 d_0 \sigma = d_0 s_0 z = z.
\]
Thus $d_0 \sigma = s_0d_0f = s_0d_0d_1\sigma$, so we conclude $\sigma = s_1d_1 \sigma = s_1f$.
We have shown there is precisely one morphism $\sigma \in \oc_1 = Z_2$ with domain $f$ and codomain in $s_0(Z_0)$.
\end{proof}

\begin{proposition}
If $Z$ is an upper 2-Segal set, then $\oc = \udec(Z)$ is a unary operadic category.
\end{proposition}
\begin{proof}
If $x \in \oc_0 = Z_1$ is an object, then $\id_x = s_0x \in \oc_1 = Z_2$ satisfies $\varphi_0(\id_x) = \varphi_0 s_0 x = d_2 s_0 x = s_0 d_1 x \in s_0(Z_0)$ is a chosen local terminal object. Hence \eqref{OCA:idfiber} holds.
Likewise, a map to a chosen local terminal $\tau_x \colon x \to u$ is $s_1x$ by the proof of the preceding lemma, so $\varphi_0(\tau_x) = \varphi_0 s_1 x = d_2 s_1 x = x$. 
This establishes the first part of \eqref{OCA:uniquemap}.
If $\sigma \in \oc_1 = Z_2$ is a morphism $x\to y$, then $s_2\sigma \in Z_3 = \oc_2 = \udec (\oc)_1$ represents the morphism $\sigma \colon \tau_x \to \tau_y$, since $d_0s_2\sigma = s_1d_0\sigma = s_1y = \tau_y$ and $d_1s_2\sigma = s_1 d_1\sigma = s_1 x = \tau_x$ and $d_2 s_2 \sigma = \sigma$.
\[ \begin{tikzcd}[column sep=small]
x \ar[dr, "\tau_x"'] \ar[rr,"\sigma"] && y \ar[dl,"\tau_y"] \\
& u
\end{tikzcd} \]
Taking fibers, we have $\varphi_1 (s_2\sigma) = d_3 s_2 \sigma = \sigma$, so the second part of \eqref{OCA:uniquemap} holds.
The conclusion of \cref{lem a5 general} holds since $\varphi_{n-2} \varphi_{n-1} = d_n d_{n+1} = d_n d_n = \varphi_{n-2} d_n$.
Taking $n=2,3$ yields \eqref{OCA:fibercomp}.
\end{proof}

\ifhs
\end{appendices}
\fi

\ifhs
\addcontentsline{toc}{section}{Acknowledgements}
\section*{Acknowledgements}
\else
\subsection*{Acknowledgements}
\fi

\ifhs
This work was supported by a grant from the Simons Foundation (\#850849). 
This material is partially based upon work supported by the National Science Foundation under Grant No. DMS-1928930 while the author participated in a program supported by the Mathematical Sciences Research Institute. 
The program was held in the summer of 2022 in partnership with the Universidad Nacional Autónoma de México.
\fi
I am grateful to Michael Batanin, Joachim Kock, Justin Lynd, and Martin Markl for fruitful discussions related to the topic of this paper.

\ifhs
\addcontentsline{toc}{section}{References}
\fi

\end{document}